\documentclass[]{interact}

\usepackage{epstopdf}
\usepackage[caption=false]{subfig}

\usepackage[numbers,sort&compress]{natbib}
\bibpunct[, ]{[}{]}{,}{n}{,}{,}
\makeatletter
\def\NAT@def@citea{\def\@citea{\NAT@separator}}
\makeatother

\theoremstyle{plain}

\theoremstyle{definition}

\theoremstyle{remark}
\newtheorem{remark}{Remark}

\usepackage{amsmath}
\usepackage{xcolor}
\usepackage{bm}

\usepackage{url}

\newtheorem{theo}{Theorem}
\newtheorem{lemm}{Lemma}

\newtheorem{rem}{Remark}

\begin{document}


\title{Existence and uniqueness of solutions for Leontief's Input-Output Model, graph theory and sensitivity analysis}


\author{
\name{Jos\'e Carlos Bellido\textsuperscript{a}\textsuperscript{b} and Luis Felipe Prieto-Mart\'inez\textsuperscript{c}\thanks{CONTACT L.~F. Prieto-Mart\'inez. Email: luisfelipe.prieto@upm.es}}
\affil{\textsuperscript{a}Departamento de Matem\'aticas, ETSI Industriales, Universidad de Castilla-La Mancha, Campus Universitario S/N, E13071, Ciudad Real, Spain; \textsuperscript{b} INEI, Universidad de Castilla-La Mancha, Campus Universitario S/N, E13071, Ciudad Real, Spain; \textsuperscript{c}Departamento de Matem\'atica Aplicada, Universidad Polit\'ecnica de Madrid, Juan de Herrera 4, Madrid 28040, Madrid, Spain}
}



\maketitle

\begin{abstract} We provide a complete study of existence and uniqueness (uniqueness up to multiples in the case $\bm d=\bm 0$) of non-negative and non-trivial solutions $\bm x$ for the linear system $(\bm I-\bm A)\bm x=\bm d$ with $\bm A\geq \bm 0,\bm d\geq 0$ (which, in particular, applies to Leontief's Input Output Model). This study is done in terms of the block triangular form of the matrix $\bm A$ and is related to a directed graph associated to $\bm A$. In particular, this study of existence and uniqueness up to multiples of solutions provides a framework that allows us to rigorously perform a sensitivity analysis of the normalized solutions of the linear system $(\bm I-\bm A)\bm x=\bm d$.
\end{abstract}


\begin{amscode} 91B66; 15A42; 90C31\end{amscode}

\textbf{Word count:} 4839 (text) + 68 (headers) = 4907 (total).



\section{Introduction}

In this paper, for a vector or matrix $\bm C$, we use the notation $\bm C\geq \bm 0$ or we say that $\bm C$ is non-negative (resp. $\bm C>\bm 0$ or we say that $\bm C$ is positive) to express that every entry in $\bm C$ is greater or equal than $0$ (resp. positive).

\subsection*{Leontief's Input-Output Model}

In classical economics, the interrelations between industries with respect to the production are crucial to understand an economic system. W. Leontief made a great contribution to this field by developing the so called Input-Output analysis. He also developed a model for describing these situations (see Chapter 10 in the book \cite{AEP} for basic information concerning this topic or the book by Leontief \cite{L}).


\medskip

\noindent \textbf{Leontief's Input Output Model} \emph{Suppose that we have an economy with a set of $n$ sectors $\mathcal S=\{S_1,\ldots, S_n\}$. Each sector receives an input and uses it to produce $x_i$ units of a single homogeneous good in a fixed period of time.  For each $i,j=1,\ldots,n$, assume that the sector $S_j$, in order to produce 1 unit, must use $a_{ij}$ units from sector $S_i$. Furthermore, assume that each sector $S_i$ sells part of its output to other sectors  and the rest, denoted by $d_i$, to consumers. The problem is to determine the non-negative and non-trivial vectors $\bm x=[x_1,\ldots,x_n]^T$ such that, for $\bm A=[a_{ij}]_{1\leq i,j\leq n}$, $\bm d=[d_1,\ldots,d_n]^T$, we reach an equilibrium}
\begin{equation} \label{eq.main1}\bm (\bm I-\bm A)\bm x=\bm d. \end{equation}

The case in which there is no external demand ($\bm d=\bm 0$) is called \textbf{closed case} and the case in which there is some external demand ($\bm d\geq 0$, $\bm d\neq \bm 0$) is called \textbf{open case}.


We say  that a solution $\bm x$ of System \eqref{eq.main1} is \textbf{economically meaningful} if all of its entries are positive, that is, $\bm x>\bm 0$.


\subsection*{Existence and Uniqueness}

The problem of studying existence and uniqueness of solutions of System \eqref{eq.main1} has already been considered. But the literature usually focus in the case of economically meaningful solutions (in the economical context if, for some $i=1,\ldots, n$ the good $i$ is not necessary for production, that is $x_i=0$, the Sector $S_i$ is removed) in the case of $\bm A$ being irreducible. We refer the reader to Section 2 in \cite{AEP}, Chapters 1 and 2 in \cite{S} and to the survey \cite{G} for more information\footnote{This paper \cite{G} is, up to our knowledge, not published in any scientific journal, but we would like to point out that in the authors' opinion it is a very remarkable, complete, interesting and inspiring work.}.


In our opinion,  it would also be interesting to perform a complete study of existence and uniqueness for the case in which the solutions are only requested to be non-negative and non-trivial and the matrix may fail to be irreducible (see Section 2), since Leontief's Input-Output Model describes a wide range of processes for which these conditions may be unrealistic. For instance, we can take a look at the Input-Output tables corresponding to the annual Spanish national accounts of the year 2016 (see \cite{INE}) to see that the matrix of technical coefficients $\bm A$ is not irreducible and the system $(\bm I-\bm A)\bm x=\bm d$ for $\bm d$ being the vector containing the outputs at basic prices fails to have an economically meaningful solution (this will be explained in more detail in the last example appearing in this paper).

Some attempts to clarify these questions have already appeared in the literature but, as far as we know, only partial results have been achieved in this direction. For instance, for the closed model, in \cite{Z} the author shows necessary and sufficient conditions
for a non-negative matrix $\bm A$ to have a unique positive (right) eigenvector $\bm x$ using spectral theory, and the
uniqueness and the reducible case are emphasized. Moreover, the questions that we intend to answer have already been discussed for the more complicated dynamic input-output model (we refer again to \cite{Z}) or in the studies of economies with environmental protection, where some of the entries in the vector $\bm d$ may fail to be non-negative (see \cite{Lee, LB}).




\subsection*{Relation to Graph Theory}

The study of these existence and uniqueness questions is done in terms of the block triangular form of the matrix $\bm A$ which is, in turn, closely related to some connectivity properties of the weighted digraph $G_{\bm A}$ with adjacency matrix $\bm A$.


These weighted graphs, as described in the following sections, serve to express the ideas of the block triangular form in other terminology. But also it is possible to leverage the information of the combinatorial structure of the underlying graph to draw some conclusions concerning existence and uniqueness. For instance, the study of existence and uniqueness of solutions presented here can be done in terms of properties of connection between the vertices and the matrices corresponding to the closures of the corresponding graphs. See also Remarks \ref{rem.1} and \ref{rem.2} for the explanation of some simple situations.


\subsection*{Sensitivity analysis}

 Roughly speaking, sensitivity analysis refers to the family of techniques utilized for evaluating the impact of uncertainties of one or more input variables of a mathematical model on the output variables. It is of great relevance in many applied mathematical contexts (see \cite{C, K, R, RBD, ST, YDBR} and the references therein).

The sensitivity analysis of this model has been widely developed in the literature (see for instance \cite{A, BS, YZ, E,  TR}), and its importance will be discussed in Section \ref{app}.

The characterization of the matrices $\bm A$ for which $(\bm I-\bm A)\bm x=\bm d$ (resp. $(\bm I- \bm A)\bm x=\bm 0$) has a unique solution (resp. a unique solution up to multiples) would provide a rigorous framework to perform sensitivity analysis. This characterization complements the discussion made in the recent preprint \cite{BP} in order to obtain a rigorous mathematical framework to perform sensitivity analysis (we refer the reader to \cite{BP} for a discussion concerning the importance of ensuring uniqueness when performing sensitivity analysis in this type of problems).

\subsection*{What do we do in this article?}

In Section \ref{section.pibt} we introduce the terminology required for the rest of this article and in Section \ref{section.lemmas} we prove some very basic lemmas (based on well known results appearing in the literature). We would like to highlight Lemma \ref{thelemma} which explains, in our approach, the relation between the matrix $\bm A$ which is called \emph{matrix of technical coefficients} in the bibliography and the so called \emph{transaction matrix}.

In Section \ref{closed} we characterize the block triangular form of the non-negative square matrices $\bm A$ for which non-trivial and non-negative solutions $\bm x$ of System $(\bm I-\bm A)\bm x=\bm 0$ exist (Theorem \ref{closed.existence}) and are unique up to multiples (Theorem \ref{closed.uniqueness}). Similarly, in Section \ref{open} we characterize (the block triangular form of the) non-negative square matrices $\bm A$ for which non-trivial and non-negative solutions $\bm x$ of System $(\bm I-\bm A)\bm x=\bm d$, for given $\bm d\neq \bm 0$,  exist Theorems \ref{open.existence} and \ref{open.uniqueness}. In both cases we relate this characterization with connectivity properties of the underlying weighted digraph $G_{\bm A}$.

The proofs of the main results appearing in Sections \ref{closed} and \ref{open} are derived from the basic lemmas in Section \ref{section.lemmas}. For us, the most important thing is to provide a complete study of these existence and uniqueness questions (related to different cases, such as reducible or irreducible matrices and positive or non-negative solutions) from the most simple and unified possible approach. We have not found such a complete and simple study in the literature, although some related results appear in \cite{AEP, G}. Moreover, Sections \ref{section.lemmas}, \ref{closed} and \ref{open} are self-contained, except for the Perron-Frobenius Theorem and the characterization of productive matrices cited in Section \ref{section.lemmas} (which are already classical results). {Also, such results are crucial in order to give a rigorous sensitivity analysis as we will see.}

Finally, in Section \ref{app} we discuss the importance of the sensitivity analysis concerning this problem and we provide some references related to this. We show some concrete examples of systems of the type $(\bm I-\bm A)\bm x=\bm d$ where the entries in $\bm A$ depend on parameters (or are parameter themselves). We prove rigorously that the solution, in each case, is unique (up to multiples in the case $\bm d=\bm 0$). After this, we develop an a direct method for computing the sensitivities, indeed the elasticities, a meaningful magnitude from an Economy viewpoint  proportional to the sentivities, of the solution $\bm x$ with respect to the entries of $\bf A$ and $\bm d$. Our method here follows the ideas in \cite{BP}, a general study by the authors on sensitivity of solutions of linear systems with respect to system coefficients, covering the general situation of manifolds of solutions of any dimension (here, we are only interested in solutions which are unique, in the open case, or unique up to multiples, in the closed case, therefore manifolds of solutions of dimension less or equal one). We would like to emphasize that the method we provide for sensitivity computation is accurate and computationally efficient in contrast with the majority of the methods implemented in the existing references on the topic.


\section{Some basic concepts} \label{section.pibt}

\subsection*{Productive matrices}

We say that a square matrix $\bm A$ is a \textbf{productive matrix} if it is non-negative and there exists some positive vector $\bm x$ such that $(\bm I-\bm A)\bm x$ is a {positive} vector. 

 There is a well-known characterization of productive matrices. For a non-negative matrix $\bm A$  the following statements are equivalent: (A) $\bm A$ is productive, (B) $(\bm I-\bm A)$ is invertible and its inverse is non-negative and (C) $\bm A$ satisfies the so called \textbf{Hawkins-Simon condition}, that is, all leading principal minors of the matrix $\bm I-\bm A$ are positive  (see Theorem 10.6 in \cite{AEP}).

\subsection*{Reducible matrices and block triangular form}

We say that a square matrix $\bm A$ is \textbf{reducible} if there exists a permutation matrix $\bm P$ such that 
\begin{equation}\label{eq.reducible} \bm P\bm A\bm P^{-1}=\left[\begin{array}{c | c} \bm F & \bm G\\ \hline \bm 0 & \bm H \end{array}\right],\end{equation}

\noindent where $\bm F,\bm H$ are square matrices (of size greater than zero). A matrix is \textbf{irreducible} if it is not reducible. There exist some known characterizations of irreducible matrices (see Lemmas 10.7, 10.8 in \cite{AEP}).

A mathematical fact is that any square matrix $\bm A$  admits a \textbf{block triangular form} (see \cite{DER}), that is, there exists a permutation matrix $\bm P$ such that
\begin{equation} \label{eq.tf} \bm P\bm A\bm P^{-1}=\left[\begin{array}{c|c|c} \bm A_{11} &\hdots & \bm A_{1k}\\ \hline \bm 0& \ddots & \vdots \\ \hline \bm 0& \bm 0&\bm A_{kk}\end{array}\right]  \end{equation}

\noindent  and the blocks $\bm A_{11},\ldots, \bm A_{kk}$ are irreducible (the case $k=1$ corresponds to $\bm A$ being irreducible). 
{This block triangular form is sometimes called canonical form (see \cite{S}), \emph{Frobenius canonical form} or \emph{irreducible canonical form} and it is \emph{essentially unique} (in some cases the ordering of the blocks is not unique and the same happens for the ordering of the rows/columns within each block, see again \cite{S}).  For a given square matrix $\bm A$, the computational problem of finding the block triangular form have also been studied. We refer the reader to \cite{DER} and to the references therein for more information.}

\subsection*{Easiest case of Equation \eqref{eq.main1}}

The most studied case in the literature is the one appearing in the following theorem.

\begin{theo}[Theorem 10.9 in \cite{AEP}] For $\bm A\geq \bm 0$, the following three conditions are equivalent:

\begin{itemize}

\item for each non-zero $\bm d\geq \bm 0$, System \eqref{eq.main1} has a positive solution $\bm x>\bm 0$;

\item the matrix $(\bm I-\bm A)^{-1}$ is positive;

\item the matrix $\bm A$ is productive and irreducible.

\end{itemize}

\end{theo}

So, in this paper, we are interested in extending this result to cover the rest of the cases described in the introduction (reducible or non-productive matrix $\bm A$ and non-negative solutions $\bm x$).


\subsection*{Weighted digraphs} 

Let $\bm A$ be a  non-negative $n\times n$ matrix. We define $G_{\bm A}$ to be the weighted digraph with vertices $S_1,\ldots, S_n$ which adjacency matrix is ${\bm A}$. This underlying graph has already being used in the study of multisectorial or multifactorial models (see the very complete survey \cite{G} or the article \cite{O} and the references therein). 


We need to introduce the following concepts:

\begin{itemize}

\item $S_i$ is a \textbf{direct predecessor} of $S_j$ if the edge from $S_i$ to $S_j$ is in $G_{\bm A}$. In this case, we also say that $S_j$ is a \textbf{direct successor} of $S_i$.

\item We say that a given vertex $S_i$ is a \textbf{sink} if has no outgoing edges (in other words, if the out-degree of $S_i$ is 0). On the other hand, we say that $S_i$ is a \textbf{source} if there is no incoming edges (in other words, if the in-degree of $S_i$ is 0).

\item A \textbf{closure} in $G_{\bm A}$ is a subset of vertices that have no direct successors. Note that, in the context of Leontief's Input-Output Model, if a set of sectors  $W$ form a closure this means that these sectors ``parasitize'' the rest of the sectors in the economy, in the sense that the production of the sectors in $W$ is not required by the rest of sectors in the economy but the sectors in $W$  may need part of the production of the others.

\item We say that $S_j$ is \textbf{reachable} from $S_i$ or that $S_i,S_j$ \textbf{communicate} if there is a sequence $S_{0},\ldots, S_{r}$ such that $S_{0}=S_i$, $S_{r}=S_j$ and, for every $m=1,\ldots, r$, $S_{m}$ is a direct successor of $S_{m-1}$. Two vertices are \textbf{strongly equivalent} if they are mutually reachable from each other.

\item \textbf{Strongly connected components} are the equivalence classes of vertices with respect to the strongly equivalence relation.

\item The \textbf{index} of $\bm G_A$ is the spectral radius of $\bm A$ (see \cite{CRS}).







\end{itemize}


It is easy to derive the following result. 

\begin{theo}[after Theorem 7 in \cite{G}] In the notation described above, the directed graph $G_{\bm A}$ is strongly connected if and only if $\bm A$ is irreducible.

\end{theo}

\noindent Moreover, by definition, if $\bm A$ is reducible and the matrix in Equation \eqref{eq.tf} is a triangular form of $\bm A$, then $\bm A$ has $k$ strongly connected components and the vertices corresponding (via the permutation) to the same diagonal block $\bm A_{ii}$ lie in the same strongly connected component.

By restricting the study of Leontief's Input-Output Model to irreducible matrices we are leaving apart the study of economical systems whose graph is not strongly connected and this might be an interesting case as well.

Surprisingly, the conditions for existence and uniqueness of solutions for the systems of the type $(\bm I-\bm A)\bm x=\bm d$ that we are considering can be stated in terms of direct successors and predecessors of the vertices, not in terms of reachability (as we will see soon).

\subsection*{Common Notation}

For the sake of brevity, let us define the following, which is common for Theorems \ref{closed.existence}, \ref{closed.uniqueness}, \ref{open.existence} and \ref{open.uniqueness}:


\noindent \textbf{Notation ($\star$).} \emph{Let $\bm A$ be a non-negative square matrix. Let}
$$\bm P\bm A\bm P^{-1}=\left[\begin{array}{c|c|c} \bm A_{11} &\hdots & \bm A_{1k}\\ \hline \bm 0& \ddots & \vdots \\ \hline \bm 0& \bm 0&\bm A_{kk}\end{array}\right]$$ 

\noindent \emph{be any of the block triangular forms of $\bm A$. Recall that all of the block triangular forms must have the same number of diagonal blocks $k$.}

\noindent \emph{Let $G_1,\ldots, G_k$ be the subgraphs of $G_{\bm A}$ consisting in the vertices of each of the strongly connected components of $G_{\bm A}$ and the oriented and weighted edges between them (recall that the number of strongly connected components of $G_{\bm A}$ coincides with the number of diagonal blocks in the block triangular form). We define a partition of the set of sectors $\mathcal S$ into the disjoint sets $\mathcal S_{<1},\mathcal S_{1},\mathcal S_{> 1}$, corresponding to the vertices belonging to strongly connected components whose associated graph $G_{i}$ has an adjacency matrix with spectral radius $<1$, $1$ or $>1$, respectively (each strongly connected components must be contained in only one of the sets $\mathcal S_{<1}, \mathcal S_1, \mathcal S_{>1}$).}


\section{Some preliminary lemmas} \label{section.lemmas}


Most of the effort for understanding the existence and uniqueness of solutions for the problem $(\bm I-\bm A)\bm x=\bm d$ for $\bm A,\bm d\geq 0$ consists in studying the existence of economically meaningful solutions $\bm x$ for the case in which $\bm A$ is irreducible.

In this section, we present two lemmas. The ideas contained in them are not original and already appear in the literature (compare to the similar results appearing in \cite{AEP, G}), but the precise statements, as it will be used later in this article, do not. So, for the shake of completeness we include the proof of both results in this section.

\begin{itemize}

\item Concerning the first one, it is common to assume the existence of the so called \emph{transaction matrix} (the definition may be found below) and of economically meaningful solutions (see \cite{L}, for instance). In this case, the \emph{matrix of technical coefficients} (see the following lemma) can be obtained from the transaction matrix. But our approach starts directly from the matrix of technical coefficients and does not make these assumptions. So, this situation makes this explicit statement necessary (in the rest of our investigation, it does not appear in the literature).

\item Concerning the second one, it contains a rephrasing of many other similar results (the references can be found in the proof), but written in a clear and proper way for our purposes.

\end{itemize}

\begin{lemm}\label{thelemma} Let $\bm A$ be a non-negative square matrix and $\bm d\geq \bm 0$. The following statements are equivalent.

\begin{itemize}

\item $(\bm I-\bm A)\bm x=\bm d$ has an economically meaningful solution.

\item $\bm A=\bm M\bm C$ where $\bm M=[m_{ij}]_{1\leq i,j\leq n}$, $\bm C=[c_{ij}]_{1\leq i,j\leq n}$ are non-negative square matrices such that $\bm C$ is diagonal, for every $i=1,\ldots, n$, 
$$m_{i1}+\ldots+m_{in}+d_i> 0\qquad\text{ and }\qquad c_{ii}=\frac{1}{m_{i1}+\ldots+m_{in}+d_i}.$$

\end{itemize}

\noindent Moreover, if any of the two previous statements holds, then $\rho(\bm A)\leq 1$.


\end{lemm}

\begin{proof}  Suppose that $x=[x_1,\ldots, x_n]^T$ is an economically meaningful solution, define $\bm C$ to be the $n\times n$ diagonal matrix such that, for every $i=1,\ldots, n$, $c_{ii}=\frac{1}{x_i}$ and let $\bm M=\bm A\bm C^{-1}$. We have that
$$\bm x-\bm d=\bm A\bm x=\bm M\bm C\bm x=\bm M\begin{bmatrix}1\\ \vdots \\ 1 \end{bmatrix}.$$

 Similarly, starting from the decomposition in the second statement, we can check that the  vector $\bm x=[\frac{1}{c_{11}},\ldots, \frac{1}{c_{nn}}]^T$ is an economically meaningful solution.

Finally,  using that $\|\cdot\|_{\infty}$ is consistent and sub-multiplicative and that $\|\bm C\bm M\|_\infty\leq 1$, we have that, for every $k\geq 2$, 
$$(\rho(\bm A))^k\leq \|\bm A^k\|_\infty\leq\|\bm M\|_\infty\cdot \|\bm C\bm M\|_\infty^{k-1}\cdot \|\bm C\|_\infty\leq \|\bm M\|_\infty\cdot \|\bm C\|_\infty<\infty.$$

\noindent Note that this implies that $\rho(\bm A)\leq 1$.





\end{proof}

The matrix $\bm A$ is usually called in this context \textbf{matrix of technical coefficients} and the matrix $\bm M$ appearing in the previous lemma is called \textbf{transaction matrix}.


\begin{lemm} \label{aux} Let $\bm A$ be an irreducible, square and non-negative matrix and $\bm d\geq 0$.

\begin{itemize}

\item[(a)] If the system $(\bm I-\bm A)\bm x=\bm d$ has a non-negative and non-trivial solution, then it is economically meaningful and so $\rho(\bm A)\leq 1$.

\item[(b)] Suppose that  $\rho(\bm A)=1$. Then $(\bm I-\bm A)\bm x=\bm d$ has a non-negative solution (may be $\bm 0$) if and only if $\bm d=\bm 0$. In this case, there is a unique (up to multiples) non-negative and non-trivial solution which is, additionally, economically meaningful.

\item[(c)] Suppose that $\rho(\bm A)<1$. Then $(\bm I-\bm A)\bm x=\bm d$ has a non-negative solution for every $\bm d$. This solution is always unique. This solution is non-trivial if and only if $\bm d\neq \bm 0$ and, in this case, it is economically meaningful.


\end{itemize}





\end{lemm}

\begin{proof} \fbox{Proof of (a)}  Suppose that the system has such a solution and that it is not economically meaningful. We may assume (performing a permutation) that $\bm x=[\bm x'\mid \bm 0]^T$, where $\bm x'>\bm 0$. This is a contradiction with the fact that $\bm A$ is irreducible. Let us also split in two parts the vector $\bm d=[\bm d'\mid\bm d'']^T$. Then,
$$(\bm I-\bm A)\bm x=\bm d\Longrightarrow \left[\begin{array}{c|c}\bm I-\bm F & .-\bm G\\ \hline -\bm E& \bm I -\bm H\end{array}\right]\begin{bmatrix}\bm x'\\ \hline \bm 0 \end{bmatrix}=\begin{bmatrix}\bm d'\\ \hline \bm d'' \end{bmatrix}\Longrightarrow -\bm E\bm x'=\bm d''.$$

\noindent Since the entries in $(-\bm E)$ are non-positive and the entries in $\bm d''$ are non-negative, the only possibility is that $\bm E=\bm 0$ (and $\bm d''=\bm 0$).  Since we have proved that the solution needs to be economically meaningful, we can use Lemma \ref{thelemma} and conclude that $\rho(\bm A)\leq 1$.

\noindent\fbox{Proof of (b)}   If $\bm d=\bm 0$, Perron-Frobenius Theorem ensures that there is a unique solution up to multiples (which, in fact, is economically meaningful, see Theorem 10.11 in \cite{AEP}).

On the other hand, note that, as a consequence of Perron-Frobenius Theorem (see Theorem 10.12 in \cite{AEP}), there is some positive left eigenvector $\bm z$ such that $\bm z(\bm I-\bm A)=\bm 0$. If we want the system  $(\bm I-\bm A)\bm x=\bm d$ to be compatible, then $\bm z\bm d=\bm 0$. But since $\bm z$ is positive and $\bm d$ is non-negative, the only possibility is $\bm d=\bm 0$.

\noindent\fbox{Proof of (c)} If $\rho(\bm A)<1$, then $(\bm I-\bm A)$ is invertible and its inverse is non-negative (see Theorem 10.11 in \cite{AEP}). So $\bm x=(\bm I-\bm A)^{-1}\bm d$. Obviously if $\bm d=\bm 0$, then $\bm x$ is not economically meaningful. On the other hand, if $\bm d\neq \bm 0$, then $\bm x$ is economically meaningful if and only if all the rows in $(\bm I-\bm A)$ are non-trivial and this is guaranteed if $(\bm I-\bm A)$ is invertible.




\end{proof}

We would like to remark the following:

\begin{rem}\label{rem.tf} In order to analyze existence and uniqueness of solutions of $(\bm I-\bm A)\bm x=\bm d$ we can assume without loss of generality that $\bm A$ is block triangular, since solutions of the previous system are univocally related to solutions of the system $(\bm I-\tilde{\bm A})\tilde{\bm x}=\bm P^{-1}\bm d$, with $\tilde{\bm A}$ being a block triangular form of $\bm A$ and $\bm A=\bm P\tilde{\bm A}\bm P^{-1}$, through the transformation $\tilde{\bm x}=\bm P^{-1} \bm x$.


\end{rem}

\section{Closed case} \label{closed}


The proof of Theorems \ref{closed.uniqueness} (in this section), \ref{open.existence} and \ref{open.uniqueness} (in the following one) are very similar to the proof of the following result. So the proof of the following result is more detailed and used as reference for the others.

\begin{theo}[existence, closed case] \label{closed.existence} In the context of Notation ($\star$), the following are equivalent:

\begin{itemize}

\item[(EC1)] $(\bm I-\bm A)\bm x=\bm 0$ has at least one economically meaningful solution.

\item[(EC2)] $\bm A=\bm M\bm C$ where $\bm M=[m_{ij}]_{1\leq i,j\leq n}$, $\bm C=[c_{ij}]_{1\leq i,j\leq n}$ are square non-negative matrices such that $\bm C$ is diagonal and, for every $i=1,\ldots, n$,
$$m_{i1}+\ldots+m_{in}\neq 0\qquad\text{and}\qquad c_{ii}=\frac{1}{m_{i1}+\ldots+m_{in}}. $$

\item[(EC3)] $\rho(\bm A)=1$ (in particular, this is equivalent to saying that, for $i=1,\ldots, k$, $\rho(\bm A_{ii})\leq 1$ and there is at least some $i$ for which $\rho(\bm A_{ii})=1$). Moreover,  $\rho(\bm A_{ii})=1$ if and only if $\bm A_{ij}=\bm 0$, for every $j\neq i$.

\item[(EC4)]  The index of $G_{\bm A}$ equals 1 (in particular, this is equivalent to saying that $\mathcal S_{>1}$ is empty and that $\mathcal S_1$ is not). Moreover, a strongly connected component is a closure if and only if it is contained in $\mathcal S_1$ (so there is at least 1).


\end{itemize}

\noindent Moreover, the following (weaker) statements are also equivalent:

\begin{itemize}

\item[(EC5)] The system $(\bm I-\bm A)\bm x=\bm 0 $ has a non-negative and non-trivial solution.

\item[(EC6)] There is at least some $i$ satisfying that  $\rho(\bm A_{ii})=1$ and that, if $i>1$, then for every $j$, $1\leq j<i$ such that $\rho(\bm A_{jj})\geq 1$, if any, we have that $\bm A_{ji}=\bm 0$.

\item[(EC7)]  $\mathcal S_1$ contains at least one strongly connected component with the property that the all the direct predecessors (if any) of its vertices lie in $\mathcal S_{<1}$.




\end{itemize}

\end{theo}

\begin{proof}  According to Remark \ref{rem.tf}, we may assume that  $\bm A$ is already in its block triangular form. The equivalence between (EC1) and (EC2) is already proved in Lemma \ref{thelemma}.

\noindent \fbox{Proof of (EC1)$\Leftrightarrow$(EC3)} Let us consider the systems 
$$(\bm I-\bm A_{kk})\bm x_k=\bm 0,\; \ldots \;,(\bm I-\bm A_{ii})\bm x_{i}=\bm A_{i,i+1}\bm x_{i+1}+\ldots+\bm A_{ik}\bm x_k,\; \ldots $$

\noindent The system has a solution if and only if some of these systems have. In this case, the solution is $\bm x=[\bm x_1\mid \ldots\mid \bm x_k]^T$.

\noindent On the one hand, if there exists such a solution, according to Lemma \ref{aux}, we have that
\begin{equation} \label{eq.reph}\bm A_{i,i+1}\bm x_{i+1}+\ldots+\bm A_{ik}\bm x_k\begin{cases}=\bm 0 & \text{if }\rho(\bm A_{ii})=1\\ \neq \bm 0 & \text{if }\rho(\bm A_{ii})<1 \end{cases}.\end{equation}

\noindent On the other hand, the conditions in (EC3) ensure the conditions in Equation \eqref{eq.reph}. So Lemmas \ref{thelemma} and \ref{aux} ensure the existence of the solution and
$$\bm x_i=\begin{cases}\text{the unique positive solution of }(\bm I-\bm A_{ii})\bm x_i=\bm 0 &\text{if }\rho(\bm A_{ii})=1\\  
(\bm I-\bm A_{ii})^{-1}(\bm A_{i,i+1}\bm x_{i+1}+\ldots+\bm A_{ik}\bm x_k)& \text{if }\rho(\bm A_{ii})<1\end{cases}.$$




\noindent \fbox{Proof of (EC3)$\Leftrightarrow$(EC4)} This is a direct translation between the language of the block triangular form and the language of graphs. We are using the fact that the vertices corresponding to each diagonal block correspond to the vertices in each of the strongly connected components, and the positive entries in the matrices $\bm A_{ji}$ represent the edges starting at a vertex in $G_j$ and ending at a vertex i $G_i$.

\noindent \fbox{Proof of (EC5)$\Leftrightarrow$(EC6)} It is very similar to the proof of the equivalence (EC1)$\Leftrightarrow$(EC3).

\noindent In the one hand, let us suppose that the system has such a solution $\bm x=[\bm x_1\mid\ldots\mid \bm x_i\mid \bm 0]^T$. By block multiplication, we can check the necessity of the conditions in (EC6)

\noindent On the other hand,  we can to construct a solution $\bm x=[\bm x_1\mid\ldots\mid \bm x_i\mid \bm 0]^T$ as follows:
$$\bm x_j=\begin{cases}\text{the unique positive solution of }(\bm I-\bm A_{jj})\bm x_j=\bm 0 &\text{if }j=i\\  
(\bm I-\bm A_{jj})^{-1}(\bm A_{j,j+1}\bm x_{j+1}+\ldots+\bm A_{jk}\bm x_k)& \text{if }\rho(\bm A_{jj})<1 \text{ and }j<i\\ \bm 0 & \text{in other case}\end{cases}.$$

\noindent \fbox{Proof of (EC6)$\Leftrightarrow$(EC7)} Again, this is a direct translation between the language of the block triangular form and the language of graphs.


\end{proof}

\begin{theo}[uniqueness, closed case] \label{closed.uniqueness} In the context of Notation ($\star$), the following are equivalent:

\begin{itemize}

\item[(UC1)] $(\bm I-\bm A)\bm x=\bm 0$ has a unique (up to multiples) non-trivial and non-negative solution.

\item[(UC2)] There exists a unique $i$, $1\leq i\leq k$ satisfying the following conditions:
\begin{itemize}

\item $\rho(\bm A_{ii})=1$.

\item For every $j$, $1\leq j<i$ (if any), such that $\rho(A_{jj})\geq 1$, we have that $\bm A_{ji}=0$.

\end{itemize}

\item[(UC3)] There is exactly one strongly connected component in $\mathcal S_1$ with the property that all the direct predecessor of its vertices (if any) lie in $\mathcal S_{<1}$.

\end{itemize}

\end{theo}

\begin{proof}  As in the previous result, we may assume that  $\bm A$ is already in its block triangular form.

\noindent \fbox{Proof of (UC1)$\Leftrightarrow$(UC2)} To prove the equivalence between the first two conditions, let us consider, again,  the systems 
$$(\bm I-\bm A_{kk})\bm x_k=\bm 0,\; \ldots \;,(\bm I-\bm A_{ii})\bm x_{i}=\bm A_{i,i+1}\bm x_{i+1}+\ldots+\bm A_{ik}\bm x_k,\; \ldots $$

\noindent On the one hand, let us suppose that such a solution exists. The conditions in (UC2) follow from Lemma \ref{aux}.

\noindent On the other hand, in order to both conditions to hold, the solutions need to be multiples of a given vector of the type $\bm x=[\bm x_1\mid\ldots\mid \bm x_i\mid \bm 0]^T$. This solution is the same as the one constructed in the proof of the equivalence (EC1)$\Leftrightarrow$(EC3).



\noindent \fbox{Proof of (UC2)$\Leftrightarrow$(UC3)} As in the proof of the previous theorem, the equivalence between the last two conditions follows from a direct translation between the block triangular form and the graph theory language.

\end{proof}

\begin{remark} \label{rem.1}  As a matter of example of how some direct implications can be derived from aspects concerning the connectivity between the vertices, let us suppose that, in the context of Notation ($\star$), one of the vertices, $S_r$, is a sink. In the context of Leontief's Input-Output Model, we can understand this situation as if this sector ``parasitizes'' the rest of the sectors in the economy. The out-degree of $S_r$ is 0 and so there is a row in $\bm A$ of the type $[0,\ldots, 0,\lambda,0,\ldots,0]$, where $\lambda$ is a positive number placed in the position $r$. Let us also consider the system $(\bm I-\bm A')\bm x'=\bm 0$ where $\bm A'$ is obtained from $\bm A$  by removing column $r$ and row $r$ and $\bm x'$ is obtained from $\bm x$ removing the entry in the position $r$. According to the previous theorems, we have the following situations:

\begin{itemize}

\item If $a_{rr}\neq 1$, then the entry in the position $r$ in the solution equals $0$ ($x_r=0$), so there is no economically meaningful solution in this case. In the context of Leontief's Input-Output Model, this means that the sector produces more than it needs or less than it needs and so, since there is no external demand, its production needs to be 0 to reach the equilibrium.

\item The case $a_{rr}=1$ is closely related to the open case: all the sectors need to ``feed'' this special sector and this special sector does not ``help'' the others, playing the role of an external demand. In fact, the system can be reduced to 
$$(\bm I-\bm A')\bm x'=x_r\begin{bmatrix}a_{1r}\\\vdots\\a_{r-1,r}\\a_{r+1,r}\\\vdots \\a_{nr}\end{bmatrix}\qquad\text{where }x_r\begin{bmatrix}a_{1r}\\\vdots\\a_{r-1,r}\\a_{r+1,r}\\\vdots \\a_{nr}\end{bmatrix}\neq \bm 0.$$

\item   $(\bm I-\bm A)\bm x=\bm 0$ has an economically meaningful solution if and only if $a_{rr}=1$, $\rho(\bm A')\leq 1$ and any strongly connected component is a closure if and only if it is contained in $\mathcal S_1$. In particular, if $\rho(\bm A')<1$, then $\bm A'$ is irreducible (there is no closures in $G_{\bm A'}$) and $S_r$ can be reached from some other vertex in the graph. Note that this is coherent with the results in the next section.

\item $(\bm I-\bm A)\bm x=\bm 0$ has a non-negative solution if and only if either $(\bm I-\bm A')\bm x'=\bm 0$ has a non-negative solution (the sector $S_r$ is ignored) or $S_r$ is, additionally, a source, that is, it is an isolated vertex not connected to any other vertex (which corresponds to the trivial case in which all the sectors are ignored and only $S_r$ is producing).

\item $\bm A\bm x=\bm 0$ has a unique non-negative solution if and only if either $a_{rr}\neq 1$ and $(\bm I-\bm A')\bm x'=\bm 0$ has a unique non-negative solution. In this case, the production $x_r$ is necessary 0.

\end{itemize}

\noindent Similar simplifications of the problem can be done if we consider a distinguished closure instead of a sink.

\end{remark}

\section{Open case} \label{open}







We will not include the proofs of the two following theorems, since they are very similar to the ones in the previous section.

\begin{theo}[existence, open model] \label{open.existence} In the context of Notation ($\star$), let $\bm d$ be a non-negative and non-trivial vector such that $\bm P\bm d=[\bm d_1\mid\ldots\mid \bm d_k]^T$. The following statements are equivalent:

\begin{itemize}

\item[(EO1)] The system $(\bm I-\bm A)\bm x=\bm d$ has at least one economically meaningful solution.

\item[(EO2)]  $\bm A=\bm M\bm C$ where $\bm M=[m_{ij}]_{1\leq i,j\leq n}$, $\bm C=[c_{ij}]_{1\leq i,j\leq n}$ are non-negative square matrices such that $\bm C$ is diagonal, and for every $i=1,\ldots, n$,
$$m_{i1}+\ldots+m_{in}+d_i\neq 0 \qquad\text{and}\qquad c_{ii}=\frac{1}{m_{i1}+\ldots+m_{in}+d_i}.$$

\item[(EO3)] $\rho(\bm A)\leq 1$ (in particular, this is equivalent to saying that, for $i=1,\ldots, k$, $\rho(\bm A_{ii})\leq 1$) and $\rho(\bm A_{ii})=1$ if and only if $\bm d_i=\bm 0$ and $\bm A_{ij}=\bm 0$ for every $j\neq i$.


\item[(EO4)] The index of $G_{\bm A}$ is less or equal to 1 (in particular, this is equivalent to saying that $\mathcal S_{>1}$ is empty). Moreover, a strongly connected component is contained in $\mathcal S_1$  if and only if it is a closure and the $\bm d_i$ corresponding to its vertices equals $\bm 0$.

\end{itemize}

Moreover, for $J$ being the set of indices $j$ such that $\bm d_j\neq\bm 0$ (which is non-empty), the following statements are also equivalent:

\begin{itemize}

\item[(EO5)] The system $(\bm I-\bm A)\bm x=\bm d$ has at least one non-negative and non-trivial solution.

\item[(EO6)] For every $i\in J$, $\rho(\bm A_{ii})<1$ and for every $i\notin J$, if $\rho(A_{ii})\geq 1$, then $\bm A_{ij}=\bm 0$, for every $j\in J$.

\item[(EO7)] The strongly connected components corresponding to the set $J$ are contained in $\mathcal S_{<1}$ and the strongly connected components contained in $\mathcal S_1\cup\mathcal S_{>1}$ do not have any direct predecessor in the strongly connected components corresponding to the set $J$.

\end{itemize}

\end{theo}

\begin{theo}[uniqueness, open model] \label{open.uniqueness} Let $\bm A$ be a non-negative square matrix and let $\bm d$ be a non-negative and non-trivial vector. Suppose that $\bm A$ admits a block triangular form:
$$\bm P\bm A\bm P^{-1}=\left[\begin{array}{c|c|c} \bm A_{11} &\hdots & \bm A_{1k}\\ \hline \bm 0& \ddots & \vdots \\ \hline \bm 0& \bm 0&\bm A_{kk}\end{array}\right]$$

\noindent and let $\bm P\bm d=[\bm d_1\mid\ldots\mid \bm d_k]^T$.  Let $J$ be the set of indices $j$ such that $\bm d_j\neq 0$ ($J$ es non-empty). The following statements are equivalent:

\begin{itemize}

\item[(UO1)] The system $(\bm I-\bm A)\bm x=\bm d$ has a unique non-negative and non-trivial solution.

\item[(UO2)] For every $i\in J$, $\rho(\bm A_{ii})<1$. If $i\notin J$ and $\rho(\bm A_{ii})\geq 1$, then $\bm A_{ij}=\bm 0$, for every $j\in J$. On the other hand, if $i\notin J$ and $\rho(\bm A_{ii})=1$, then there exists some $j\notin J$ such that $\rho(\bm A_{jj})\geq 1$ and $\bm A_{ji}\neq \bm 0$.


\item[(UO3)] The strongly connected components corresponding to the set $J$ belong to $\mathcal S_{<1}$. The strongly connected components contained in $\mathcal S_1\cup \mathcal S_{>1}$ do not have any direct predecessor in the strongly connected components corresponding to the set $J$. Finally in every strongly connected component contained in $\mathcal S_1$ there is some vertex with a direct predecessor in $\mathcal S_1\cup\mathcal S_{>1}$.



\end{itemize}

\end{theo}

\begin{remark} \label{rem.2} As a matter of example of how some direct implications can be derived from aspects concerning the connectivity between the vertices, let us suppose that, in the context of Notation ($\star$), one of the vertices, $S_r$, is a sink. In the context of Leontief's Input-Output Model, we can understand this situation as if this sector ``parasitizes'' the rest of the sectors in the economy but, maybe, it produces some good with an external demand. The out-degree of $S_r$ is 0 and so there is a row in $\bm A$ of the type $[0,\ldots, 0,\lambda,0,\ldots,0]$, where $\lambda$ is a positive number placed in the position $r$. Let us also consider the system $(\bm I-\bm A')\bm x'=\bm d'$ where $\bm A'$ is obtained from $\bm A$  by removing column $r$ and row $r$ and $\bm x',\bm d'$ are obtained from $\bm x,\bm d$ removing the entry in the position $r$. According to the previous theorems, we have the following situations:

\begin{itemize}

\item If $d_r> 0$, then necessarily $a_{rr}< 0$. In terms of Leontief's Input-Output Model, an external demand imposes the necessity of a (positive) production for the sector $S_r$.

\item If $d_r=0, a_{rr}\neq 1$, then $x_r=0$ is a similar fashion as this in Remark \ref{rem.1}.

\item If $d_r=0,a_{rr}=1$ this is a trivial case in which the sector $S_r$  ``parasitizes'' the rest of the sectors, it does not ``help'' them and there is no external demand of the goods produced by it. In this case, we need to study the modified problem 
$$(\bm I-\bm A')\bm x'=\bm d'+x_r\begin{bmatrix}a_{1r}\\\vdots\\a_{r-1,r}\\a_{r+1,r}\\\vdots \\a_{nr}\end{bmatrix}\qquad\text{where }\bm d+x_r\begin{bmatrix}a_{1r}\\\vdots\\a_{r-1,r}\\a_{r+1,r}\\\vdots \\a_{nr}\end{bmatrix}\neq \bm 0.$$

\item   $(\bm I-\bm A)\bm x=\bm d$ has an economically meaningful solution if and only if either $a_{rr}<1$ and $d_r> 0$ and $(\bm I-\bm A')\bm x'=\bm d'$ has an economically meaningful solution or $a_{rr}=1$, $\rho(\bm A')\leq 1$ and for every connected component in $\mathcal S_1$, this connected component is a closure and the corresponding $\bm d_i$ equals $\bm 0$. In particular, if $\rho(\bm A')<1$, then $\bm A$ is irreducible and either $\bm d'\neq 0$ or $S_r$ has at least one predecessor.

\item $(\bm I-\bm A)\bm x=\bm 0$ has a non-negative solution if and only if either $d_r=0$ and $(\bm I-\bm A')\bm x'=\bm d'$ has a non-negative solution (the sector $S_r$ is ignored) or $d_r>0, a_{rr}<1$ and the system
$$(\bm I-\bm A')\bm x'=\bm d'+x_r\begin{bmatrix}a_{1r}\\\vdots\\a_{r-1,r}\\a_{r+1,r}\\\vdots \\a_{nr}\end{bmatrix}$$

\noindent is compatible. In particular, in the last case, the vertices in $\mathcal S_1\cup\mathcal S_{>1}$ cannot reach $S_r$.

\item $\bm A\bm x=\bm 0$ has a unique non-negative solution if and only if $a_{rr}\neq 1$ and the system in the previous item has a unique non-negative solution. 
\end{itemize}

\noindent Similar simplifications of the problem can be done if we consider a distinguished closure instead of a sink.

\end{remark}

\section{Sensitivity analysis and some examples} \label{app}


Sensitivity analysis is important for determining the parameters of a system that are most critical and require more careful determination and to establish procedures that allow the solutions to be controlled. 
The predictions of input-output models depend on the accuracy of the parameters contained in the matrix of technical coefficients $\bm A$. Let us recall that, in general and in practice, the matrix of technical coefficients is determined from the transaction matrix $\bm M$ and the vector of external demand $\bm d$, which are obtained experimentally. So it seems advisable to perform sensitivity analysis to determine the effect of small changes in these parameters, since they may not be determined with complete certainty.

The sensitivity analysis of Leontief's Input-Output Model has been widely developed in the literature (see for instance 
\cite{A, BS, E}). As we have already explained, originally Leontief's Input-Output Model was used to describe the economics of a whole region in which we can find $n$ different industries each of them depending on the production of the rest.  This model is nowadays applied in a variety of contexts and, as for instance, for describing the situation of  interdependence concerning energy-related $CO_2$ emission between energy producing sectors and non-energy producing sectors. Sensitivity analysis in this last context has received a lot of recent attention (see \cite{TR,YZ}).

Let $z$ be a real variable that depends on a vector of variables $\bm \varepsilon=(\varepsilon_1,\ldots, \varepsilon_N)$ (that is, $z=z(\bm \varepsilon)$).  As used in Economy, let us recall that the \textbf{elasticity} of $z$ with respect to  $\varepsilon_i$ is defined as:
$$E_{z,\varepsilon_i}(\bm\varepsilon_0)=\frac{\varepsilon_i}{z(\bm\varepsilon_0)}\cdot \frac{\partial z}{\partial \varepsilon_i}(\bm\varepsilon_0)$$

\noindent and it is one of the possible measures of the sensitivity of $z$ with respect to each of the variables $\varepsilon_i$ near $\bm \varepsilon_0$.

The derivatives of variables with respect to parameters are not particularly meaningful from an economical perspective, as a measure of sensitivity, without significant context, and those depends significantly on the units chosen. In contrast, the elasticity has a clear meaning without context and it is dimensionless.


\noindent As far as we know, the most common approach in the literature for computing elasticities in the context of Leontief's Input-Output Model comes from the \emph{Numerical Linear Algebra viewpoint}: the parameters are perturbed with small increments and derivatives are not computed, that is, the following approximation is used (see Equation (12) in \cite{TR}):
$$E_{z,\varepsilon_i}(\bm\varepsilon)\approx\frac{\frac{\Delta z}{z}}{\frac{\Delta \varepsilon_i}{\varepsilon_i}}.$$

\noindent  In most of these works, also, to do so they need to study the sensitivity of the coefficients of \textbf{Leontief's Inverse} (that is, of the matrix $(\bm I-\bm A)^{-1}$) using the formulas appearing in \cite{SM}. Moreover, the articles cited above only deal with the open case of the model.

In the closed case it is not important to study the values of each entry of the solution as the parameters in $\bm A$ vary, but their relative size. So, it is reasonable to consider that the solutions of the closed model $(\bm I-\bm A)\bm x=\bm 0$ are always normalized $\|\bm x\|=1$. Uniqueness up to multiples correspond to uniqueness of normalized solutions and so, if the matrix of technical coefficients $\bm A$ satisfies the conditions in Theorem \ref{closed.uniqueness}, it makes sense to perform sensitivity analysis also in the closed case without differentiability issues derived from multiplicity of solutions (see \cite{BP} and the references therein).

In this section we determine the elasticities, as a measure of sensitivity, of the entries in the total production vector $\bm x$ (or of some related variables) with respect to the coefficients in $\bm A$ for both, the open economy model and the closed economy model (for which we consider normalized solutions).

\setcounter{rem}{2}

\begin{rem}\label{rem.rigor} Suppose that $\bm A=[a_{ij}]_{1\leq i,j\leq n}$. To perform the sensitivity analysis rigorously, we need to care about the two following questions.

\begin{itemize}

\item We need to ensure that the system has a unique solution (up to multiples in the closed case). The theorems required to do so have already been described in this article.

\item We need to ensure that the solution $\bm x$ is differentiable as a function of the vector of variables $(a_{11},\ldots, a_{1n},\ldots, a_{n1},\ldots, a_{nn})$. 

\end{itemize}


\end{rem}

In relation to the second item, we have the following theorem.

\begin{theo} \label{theo.nohomo} Let $\bm D(\bm\varepsilon)=[\gamma_{ij}(\bm\varepsilon)]_{1\leq i,j\leq n}$, $\bm d=[d_1(\bm\varepsilon),\ldots, d_n(\bm\varepsilon)]^T$ such that their entries are differentiable in a neighborhood $U$ of some $\bm\varepsilon_0\in\mathbb R^N$, $N\geq 1$. Suppose that
$$\forall \varepsilon\in U,\qquad\text{rank}(\bm D(\bm\varepsilon))=\text{rank}(\bm D(\bm\varepsilon)\mid \bm d(\bm\varepsilon))=r. $$

\begin{itemize}

\item[(a)] If $r=n-1$ and, for $\bm\varepsilon\in U$, $\bm d(\varepsilon)=\bm 0$, there exists some neighborhood  $U^*$ of $\bm\varepsilon_0$ such that  there exists differentiable functions $x_1(\bm\varepsilon),\ldots, x_m(\bm\varepsilon)$ such that $\bm x(\bm \varepsilon)=[x_1(\bm\varepsilon),\ldots,x_n(\bm\varepsilon)]^T$ satisfies $\bm D(\bm\varepsilon)\bm x(\bm\varepsilon)=\bm 0$ and $\|\bm x(\bm\varepsilon)\|=1$ for $\bm\varepsilon\in U^*$.

\item[(b)] If $r=n$ and, for $\bm\varepsilon\in U$, $\bm d(\varepsilon)\neq \bm 0$, there exists some neighborhood  $U^*$ of $\bm\varepsilon_0$ such that  there exists differentiable functions $x_1(\bm\varepsilon),\ldots, x_m(\bm\varepsilon)$ such that $\bm x(\bm \varepsilon)=[x_1(\bm\varepsilon),\ldots,x_n(\bm\varepsilon)]^T$ satisfies $\bm D(\bm\varepsilon)\bm x(\bm\varepsilon)=\bm d(\bm\varepsilon)$ and $\bm x(\bm\varepsilon)\neq \bm 0$ for $\bm\varepsilon\in U^*$.

\end{itemize}

\end{theo}

\begin{proof} (a) There is no {loss of generality in} assuming that $\det[\gamma_{ij}(\bm\varepsilon_0)]_{1\leq i,j\leq r})$ is a non-trivial minor.   There is a neighborhood $U^*$ of $\bm\varepsilon_0$ contained in $U$ such that, for every $\bm\varepsilon\in U^*$, $\det([\gamma_{ij}(\bm\varepsilon_0)]_{1\leq i,j\leq r})\neq 0$. Consider the matrix
$$\bm{\widetilde D}(\bm\varepsilon)=\left[\begin{array}{l l l | l }\gamma_{11}(\bm\varepsilon) & \hdots & \gamma_{1,n-1}(\bm\varepsilon) & \gamma_{1,n}(\bm\varepsilon) \\ 
\vdots & & \vdots & \vdots  \\ 
\gamma_{n-1,1}(\bm\varepsilon) & \hdots & \gamma_{n-1,n-1}(\bm\varepsilon) & \gamma_{n-1,n}(\bm\varepsilon) \\ \hline 0 & \ldots & 0 & 0  \end{array}\right].$$

\noindent For $1\leq k\leq n$, let us denote by $\Gamma_{nk}(\bm\varepsilon)$ to the cofactor of the entry in the position $(n,k)$ in the previous matrix. Define $f_k(\bm\varepsilon=\Gamma_{nk}(\bm\varepsilon) $ and consider the vector $\bm y(\bm\varepsilon)=[f_1(\bm \varepsilon),\ldots, f_n(\bm \varepsilon)]^T$. Note that $\bm y(\bm \varepsilon)$ satisfies $\bm D(\bm\varepsilon)\bm y(\bm\varepsilon)=\bm 0$ and $\bm x(\bm\varepsilon)\neq \bm 0$ for $\bm\varepsilon\in U^*$.  To see this, note that 
$$[\gamma_{i,1}(\bm\varepsilon),\ldots, ,\gamma_{in}(\bm \varepsilon)]\begin{bmatrix}f_1(\bm\varepsilon)\\ \vdots \\ f_n(\bm \varepsilon) \end{bmatrix}=\sum_{k=1}^n\gamma_{ik}(\bm\varepsilon)\Gamma_{nk}(\bm \varepsilon)=\bm 0 $$

\noindent is the determinant of the $n\times n$ matrix obtained by replacing the $n$-th row in the original one by $[\gamma_{i,1}(\bm\varepsilon),\ldots, ,\gamma_{in}(\bm \varepsilon)]$.

\noindent Finally, define $\displaystyle{\bm x(\bm\varepsilon)=\frac{1}{\|\bm y(\bm \varepsilon)\|}\bm y(\bm \varepsilon)}$.

\noindent (b) This part is a direct consequence of Cramer's Rule.




\hspace{12cm}\end{proof}

Finally, we will compute the sensintivities, and consequently the elasticities, using the approach and methods developed in \cite{BP}, making it possible to compute elasticities in a direct way, without using the \emph{Numerical Linear Algebra viewpoint} and without performing the sensitivity analysis of Leontief's Inverse.

\subsection*{Closed Model} \label{subsection.closed}

As an example, we study the following simplified and academic problem appearing in the book by Leontief \cite{L}, Chapter 2. An input-output model depicting a three-sector economy is shown in the following table, which corresponds to the transaction matrix $\bm M$.  The sectors $S_1, S_2, S_3$ represent agriculture,  manufacture and households, respectively.


\begin{table}[h!]
\tbl{ }
{ \begin{tabular}{|l| l  l l | l|l|}
\hline      &$S_1$ &$S_2$&$S_3$  & Total Production ($\bm x_0$)& unit\\
\hline $S_1$ &  25& 20& 55& 100 &bushels of wheat\\
$S_2$& 14&  6& 30& 50 &yards of cloth\\
$S_3$& 80& 180& 40& 300 &man-years of labor \\
\hline 
 \end{tabular}
}

\end{table}



\noindent Let $\bm x_0=[100, 50, 300]^T$be the vector of total productions and let $\bm C$ be the diagonal matrix whose diagonal entries are $\frac{1}{100},\frac{1}{50},\frac{1}{300}$. Then the matrix of technical coefficients is
$$\bm A_0=\bm M\bm C=\begin{bmatrix}{25}/{100} & {20}/{50}& {55}/{300}\\ {14}/{100} & {6}/{50} & {30}/{300}\\ {80}/{100} & {180}/{50} & {40}/{300} \end{bmatrix}\approx\begin{bmatrix}0.25 & 0.4& 0.1833\\ 0.14 & 0.12 & 0.1\\0.8 & 3.6 & 0.1333 \end{bmatrix}.$$

\noindent See that the vector or total production $\bm x_0$ is an economically meaningful solution of the homogeneous system:
$${(\bm I-\bm A_0)}\bm x_0={\begin{bmatrix}{75}/{100} & -{20}/{50}& -{55}/{300}\\ -{14}/{100} & {44}/{50} & -{30}/{300}\\ -{80}/{100} & -{180}/{50} & {260}/{300} \end{bmatrix} }\begin{bmatrix}100\\ 50\\ 300 \end{bmatrix}=\begin{bmatrix} 0 \\ 0 \\ 0 \end{bmatrix}.$$

Consider the matrix $\bm A(\bm \varepsilon)=[a_{ij}]_{1\leq i,j\leq 3}$, whose entries are considered as variables, that is, $\bm\varepsilon=(a_{11},a_{12},a_{13},a_{21},\ldots, a_{23},a_{33})$ and let $\bm B(\bm \varepsilon)=(\bm I-\bm A(\bm\varepsilon))$. Let us consider the system $\bm B(\bm \varepsilon)\bm x(\bm \varepsilon)=\bm 0$, for $\bm x(\bm\varepsilon)=[x_1(\bm\varepsilon),x_2(\bm \varepsilon), x_3(\bm \varepsilon)]^T$. Let 
$$\bm\varepsilon_0=\left(\frac{25}{100}, \frac{20}{50}, \frac{55}{300}, \frac{14}{100}, \frac{6}{50},\frac{30}{300},\frac{80}{100}  \frac{180}{50}, \frac{40}{300}\right).$$

\noindent  Note that $\bm A(\bm\varepsilon_0)=\bm A_0$ and $\bm x(\bm\varepsilon_0)=\bm x_0$. We want to solve the following

\medskip

\noindent \textbf{Problem.} \emph{Compute the elasticities of the variables $x_1,x_2,x_3$ with respect to the technical coefficients $\bm \varepsilon=(a_{11},a_{12},\ldots,a_{23},a_{33})$ in a neighborhood of $\bm \varepsilon_0$.}

\medskip

As explained in Remark \ref{rem.rigor}, to perform sensitivity analysis of this problem we need to check two things. The first one is to ensure that the system has a unique non-negative solution up to multiples. This is true, since the matrix $\bm A$ is irreducible in a neighborhood of $\bm \varepsilon_0$ (Theorem \ref{closed.uniqueness}). The second one is to ensure that the entries in $\bm x$ are differentiable with respect to the vector of variables $\bm\varepsilon=(a_{11},a_{12},a_{13},\ldots, a_{31},a_{32},a_{33})$ (this is also true since the entries in $\bm B(\bm \varepsilon)$ are differentiable with respect to the vector of variables $\bm \varepsilon$ and the previous condition holds, as ensured by Theorem \ref{theo.nohomo}).


Now, to solve the problem, we need to compute the partial derivatives $\frac{\partial \bm x}{\partial a_{ij}}(\bm\varepsilon_0)$. We can use a modification of Nelson's Method (\cite{N}) as appears in the following table.  Nelson's Method is one of the most efficient algorithms for the sensitivities computation with respect to matrix parameters of eigenvectors associated to simple eigenvalues, and in \cite{BP} we propose an adaptation of Nelson's method to sensitivities of linear systems solutions, which is suitable for sensitivities computation in the closed model case (more comments about this are made in Subsection 4.1 of \cite{BP}).

\medskip

\begin{table}[h!]
\tbl{ }
{
\begin{tabular}{|c l|}\hline & \textbf{Direct Method}\\ \hline 1.- & We look for a particular solution $\bm v$ of the linear system \\
& $\displaystyle{(\bm I-\bm A_0)\cdot \bm v=x_j(\bm \varepsilon_0)\cdot \bm e_i}$ \\
& where $\bm e_j$ denotes the vector which entries are 1 in the position $j$ and 0 in the rest.\\
 2.- & Set $c=-\bm v^T\bm x_0$. \\ 3.- & Then $\displaystyle{\frac{\partial \bm x}{\partial a_{ij}}(\bm \varepsilon_0)=\bm v+c\bm x_0}$.\\ \hline

\end{tabular}
}
\end{table}

\medskip

\noindent Let us remark that in Subsection 4.2 in \cite{BP} we can also find an adjoint method which is more efficient for large values of $n$.

\medskip

  We obtain the following Jacobian matrix ($\frac{\partial x_m}{\partial a_{ij}}(\bm\varepsilon_0)$ is the entry $m,3(i-1)+j$):
$$\begin{bmatrix}   95.0681  & 47.5341  &285.2043 & -34.4990 & -17.2495 &-103.4970 & -17.6412  & -8.8206 & -52.9236\\
  -25.8655 & -12.9328 & -77.5965  & 14.4834   & 7.2417  & 43.4502  &-20.6859&  -10.3430  &-62.0578\\
  -27.3785 & -13.6892 & -82.1354&    9.0858 &   4.5429  & 27.2573&    9.3281&    4.6640&   27.9842\end{bmatrix}$$

\noindent and then the following matrix containing the elasticities ($E_{x_m,a_{ij}}(\bm\varepsilon_0)$ is the entry $m,3(i-1)+j$):
$$\begin{bmatrix} 0.2377 &   0.1901  &  0.5229 &  -0.0483  & -0.0207 &  -0.1035  & -0.1411  & -0.3175 &  -0.0706\\
   -0.1293 &  -0.1035 &  -0.2845 &   0.0406 &   0.0174   & 0.0869 &  -0.3310  & -0.7447  & -0.1655\\
   -0.0228  & -0.0183 &  -0.0502 &   0.0042  &  0.0018 &   0.0091  &  0.0249  &  0.0560  &  0.0124\end{bmatrix}.$$

Note that the greatest values are, by far, the entries in the positions $(1,3)$ and $(2,8)$ in the matrix above, corresponding to the total production of agriculture with respect to $a_{13}$ (number of bushels of wheat required per one man-year of labor) and the total production of manufacture with respect to $a_{32}$ 
(man-years of labor required to produce one yard of cloth).

\subsection*{Open Model (first example)}

In \cite{TR} the authors perform the sensitivity analysis of a problem concerning energy-related $CO_2$ emission. In the following, we will carry out the sensitivity  analysis for the same problem, but using our own approach. We would like to remark again that in the literature (including \cite{TR}) the sensitivity analysis is traditionally done in a less direct way (studying the sensitivity of the coefficients of Leontief's Inverse).

\medskip


Let us consider an economic system with four economic sectors: $S_1,S_2$ are two energy producing sectors and $S_3,S_4$ are two non-energy sectors. Each productive sector consumes energy that is generated by an energy sector. The production relationships between them are captured in the following table (corresponding to the transaction matrix $\bm M$).

\begin{table}[h!]
\tbl{ }
{
 \begin{tabular}{|l| l l l l |l | l|}
\hline      &$S_1$ &$S_2$ &$S_3$ &$S_4$ & $\bm d$ & $\bm x_0$\\
\hline $S_1$& 174& 255& 347& 44& 50&870\\
$S_2$& 87& 102& 139& 132& 50&510\\
$S_3$& 87& 51& 70& 88& 400 &696\\
$S_4$& 87&51 &70 &132 &100 &440\\
\hline 
 \end{tabular}}
\end{table}

\noindent The rows in the previous table show the sales of the sectors, while the columns show their purchases, in million euros. Vector $\bm d$ contains the final demands of each sector and vector $\bm x_0$ the total sales (including demand). Let $\bm z_0=[z_1,z_2]^T$ be the vector containing the quantity of $CO_2$ emission by the energy producing sectors $S_1$ and $S_2$. Let us assume that the quantity of $CO_2$ emission generated by each sector can be decomposed into the following two factors:
$$z_m=c_m\cdot x_m,\qquad \text{for }m=1,2, $$

\noindent where $c_m$ represents the \textbf{intensity coefficient} (emission per unit of output of sector $i$). For this problem, let us take $c_1=1,c_2=5$. The matrix of technical coefficients is
$$\bm A_0=\begin{bmatrix}0.2 & 0.5& 0.5 & 0.1\\ 0.1 & 0.2 & 0.2 & 0.3\\ 0.1 & 0.1 & 0.1 & 0.2\\ 0.1 & 0.1 & 0.1 & 0.3 \end{bmatrix}. $$

Let us consider, in a similar way to the one followed in the previous subsection, that the entries in the matrix $\bm A(\bm \varepsilon)=[a_{ij}]_{1\leq i,j\leq 4}$ are variables, that is, $\bm \varepsilon=(a_{11},a_{12},\ldots, a_{43},a_{44})$. We consider that the entries in the vector $\bm d$ and the values $c_1,c_2$ has been determined with precision and we let them out of the analysis (they could also be included effortlessly as it is explained in the following example).  Let $\bm B(\bm \varepsilon)=(\bm I-\bm A(\bm \varepsilon))$. Let us consider the system $\bm B(\bm \varepsilon)\bm x(\bm \varepsilon)=\bm d$ and the variables $z_1(\bm \varepsilon)=c_1\cdot x_1(\bm \varepsilon),z_2(\bm \varepsilon)=c_2\cdot x_2(\bm \varepsilon)$, where $x_j(\bm \varepsilon)$, for $j=1,\ldots, 4$, denotes the corresponding entry in $\bm x(\bm \varepsilon)$. For
$$\bm \varepsilon_0=(0.2 , 0.5, 0.5 , 0.1, 0.1 , 0.2 , 0.2 , 0.3,0.1 , 0.1 , 0.1 , 0.2, 0.1 , 0.1 , 0.1 , 0.3)$$

\noindent  we have, again, that $\bm A(\bm \varepsilon_0)=\bm A_0$ and $\bm x(\varepsilon_0)=\bm x_0$. This time we are interested in the following

\medskip

\noindent \textbf{Problem.} \emph{Compute the elasticity of the variables $z_1, z_2$ with respect to the technical coefficients  $\bm \varepsilon=(a_{11},a_{12},\ldots,a_{43},a_{44})$ in a neighborhood of $\bm \varepsilon_0$.}

\medskip

As explained in Remark \ref{rem.rigor} and happened in the previous subsection, to perform sensitivity analysis of this problem we need to ensure that the system has a unique non-negative solution (which is again true, according to Theorem \ref{open.uniqueness}, since the matrix $\bm A$ is irreducible) and to ensure that the entries in $\bm x$ are differentiable with respect to the vector of variables $\bm \varepsilon=(a_{11},a_{12},a_{13},\ldots, a_{31},a_{32},a_{33})$ (which is true according to Theorem \ref{theo.nohomo}).

In this case, we can just compute the partial derivatives as follows (it can be proved straightforward, but the reader may find a larger discussion in Lemma 18 in \cite{BP}):
$$\bm B(\bm \varepsilon_0)\frac{\partial \bm x}{\partial a_{ij}}(\bm \varepsilon_0)=x_j(\bm \varepsilon_0)\cdot \bm e_{i} $$

\noindent  where $\bm e_i$ denotes the vector which entries are 1 in the position $j$ and $0$ in the rest of them. So we can easily obtain the Jacobian matrix to our problem. We represent its transpose below ($\frac{\partial z_m}{a_{ij}}(\partial\bm\varepsilon_0)$ correspond to the entry $3(i-1)+j,m$):
$$\begin{bmatrix}   1.5432  &  0.3823    \\
    0.9046  &  0.2241    \\
    1.2345   & 0.3059    \\
    0.7804    &0.1934    \\
    1.2623    &1.5189    \\
    0.7400    &0.8904    \\
    1.0099    &1.2152    \\
    0.6384    &0.7682    \\
    1.2611    &0.6488    \\
    0.7393    &0.3804    \\
    1.0089    &0.5191    \\
    0.6378    &0.3282    \\
    1.1218    &0.8910    \\
    0.6576    &0.5223  \\
    0.8974    &0.7128 \\
    0.5673    &0.4506\end{bmatrix}.$$

Let us represent the values of the elasticities ($E_{e_1,a_{ij}}$, $E_{e_2,a_{ij}}$) in a table, instead of a matrix, to make it easier to compare with the approximate values ($\varepsilon_{e_{e1}a_{ij}}$,  $\varepsilon_{e_{e2}a_{ij}}$) appearing in Table 5 in \cite{TR}. We have obtained approximately the same values for the elasticity that in that article.

\begin{table}[h!]
\tbl{ }
{
 \begin{tabular}{|l l | l l|l l |}
\hline i & j &$E_{e_1,a_{ij}}$ & $E_{e_2,a_{ij}}$& $\varepsilon_{e_{e1}a_{ij}}$ &  $\varepsilon_{e_{e2}a_{ij}}$ \\
\hline 1&1 & 0.3547 & 0.1499 &0.3561 &  0.1506\\ 
1& 2&  0.5199  &0.2197&0.5207 &  0.2202\\
1&3 &0.7075 &0.2990& 0.7096&  0.3002\\
1&4  &0.0897 &0.0379&0.0896 &0.0379\\
2&1  & 0.1451 &0.2978&0.1454 &  0.2986\\
2& 2 & 0.1701 &0.3492&0.1706  &0.3504 \\
2&3  & 0.2318 &0.4758&0.2321  &0.4766 \\
2& 4 & 0.2201&0.4519&0.2201  &0.4520 \\
3&1  &0.1450 &0.1272&0.1454  &0.1276 \\
3&2  &0.0850 &0.0746&0.0851 & 0.0747\\
3&3  &0.1166 &0.1024& 0.1161 &0.1019 \\
3&4  & 0.1466 &0.1287&0.1467  &0.1287\\
4&1  &0.1289 &0.1747&0.1292  &0.1751 \\
4&2  & 0.0756&0.1024&0.0756  &0.1025 \\
 4& 3 &0.1037 & 0.1406&0.1031 & 0.1398\\
4& 4& 0.1956&0.2651 &0.1964  &0.2662 \\ 
\hline
 \end{tabular}}
\end{table}

\subsection*{Open Model (second example)}

Let us take, for this example, the $65\times 65$ matrix of technical coefficients $\bm A_0$ appearing in \cite{INE} and let us consider the vector (or vectors) $\bm x_0$ satisfying
$$(\bm I-\bm A_0)\bm x_0=\bm d_0, $$

\noindent where $\bm d_0$ is the column containing the outputs at basic prices. Sectors 44a (imputed rents of owner-occupied dwellings), 47 (scientific research and development services), 63 (services of households as employers: undifferentiated goods and services produced by households for own use) and 64 (services provided by extraterritorial organizations and bodies) are special. Using the pertinent permutation matrix $\bm P$, we reduce our system to one of the type:
$$\underbrace{\left[\begin{array}{l | l l l l}\bm I-\bm A_{11} &-\bm A_{1,44a} & -\bm A_{1,47}& -\bm A_{1,63}& -\bm A_{1,64}\\ \hline  \bm 0& 1& 0& 0& 0\\ \bm 0& 0& 0.97673& 0& 0\\ \bm 0& 0& 0& 1&0\\ \bm 0& 0& 0& 0&1 \end{array}\right]}_{\bm I-\bm P\bm A_0\bm P^{-1}}\begin{bmatrix}\widetilde{\bm x}_1\\ \hline x_{44a}\\ x_{47}\\ x_{63} \\ x_{64} \end{bmatrix}=\begin{bmatrix}\widetilde{\bm d}_1\\ \hline  90119\\ 16931.2 \\ 9761 \\ 0 \end{bmatrix}. $$

\noindent Let us remark that, in this example, the blocks $\bm A_{1,44a},\bm A_{1,47},\bm A_{1,63}, \bm A_{1,64}$ are not $\bm 0$. $\bm A_{11}$ is an irreducible block and so are the $1\times 1$ blocks $\bm A_{44a,44a}=[0]$, $\bm A_{47,47}=[0.02327]$, $\bm A_{63,63}=[0]$, $\bm A_{64,64}=[0]$.

In a similar fashion to the one in the previous example, let us consider that the vector $\bm \varepsilon$ of length $65^2+65$ that represent the entries in the matrix of technical coefficients and in the external demand vector.  Let $\bm \varepsilon_0$ be the value of the variables corresponding to $\bm A_0,\bm d_0$. Let $\widetilde{\bm B}(\bm \varepsilon)=(\bm I-\bm P\bm A(\bm \varepsilon)\bm P^{-1})$ and $\widetilde{\bm d}(\bm\varepsilon)=\bm P\bm d(\bm\varepsilon)$. Let us consider the system $\widetilde{\bm B}(\bm \varepsilon)\widetilde{\bm x}(\bm \varepsilon)=\widetilde{\bm d}(\bm\varepsilon)$. We are interested in the following:

\medskip

\noindent \textbf{Problem.} \emph{Compute the elasticity of the terms in $\widetilde{\bm x}(\bm\varepsilon)$ with respect to the elements in $\bm \varepsilon$.}

\medskip

As we have explained, to perform sensitivity analysis of this problem we need to ensure that the system has a unique non-negative solution (which is again true, according to Theorem \ref{open.uniqueness}) and to ensure that the entries in $\widetilde{\bm x}$ are differentiable with respect to the vector of variables $\bm \varepsilon$ (which is true, according to Theorem \ref{theo.nohomo}). We have that, for any variable $\varepsilon_k$ in the vector $\bm\varepsilon$:
\begin{equation} \label{eq.derivada}\widetilde{\bm B}(\bm \varepsilon_0)\frac{\partial \widetilde{\bm x}}{\partial \varepsilon_i}(\bm \varepsilon_0)=\begin{cases}\widetilde x_n(\bm \varepsilon_0)\cdot \bm e_m&\text{if $\varepsilon_i$ corresponds to the position $(m,n)$}\\
&\text{in the matrix matrix of tech. coeff.}\\
 \displaystyle{\bm e_m} & \text{if $\varepsilon_i$ corresponds to the position $m$ in the}\\
&\text{ final demand vector,} \end{cases} \end{equation}

\noindent where $\bm e_m$ denotes the vector which entries are 1 in the position $m$ and 0 otherwise and $\widetilde x_n(\bm \varepsilon_0)$ denotes the entry in the position $n$ of $\widetilde{\bm x}(\bm \varepsilon_0)$.

 Note that, in this case, it is very easy to compute the derivative $\displaystyle{\frac{\partial \widetilde x_n}{\partial a_{ij}}(\varepsilon_0)}$ for $n=44a, 47, 63, 64$. This is part of a more general principle:

\begin{rem} Let us consider the general reducible case:
$$\left[\begin{array}{c|c|c} \bm I-\bm A_{11}(\bm \varepsilon_0) &\hdots & \bm -\bm A_{1k}(\bm\varepsilon_0)\\ \hline \bm 0& \ddots & \vdots \\ \hline \bm 0& \bm 0&\bm I-\bm A_{kk}(\bm\varepsilon_0)\end{array}\right]\begin{bmatrix}\widetilde{\bm x}_1(\bm\varepsilon_0)\\ \hline \vdots \\ \hline \widetilde{\bm x}_k(\bm \varepsilon_0) \end{bmatrix}=\begin{bmatrix}\widetilde{\bm d}_1(\bm \varepsilon_0)\\ \hline \vdots \\ \hline  \widetilde{\bm d}_k(\bm \varepsilon_0) \end{bmatrix}.$$

\noindent Note that, for any index $i$ corresponding the block $\bm A_{jj}$, the entry in the position $i$ in $\widetilde{\bm x}(\varepsilon_0)$, $\widetilde{x}_i(\bm\varepsilon_0)$, depends only on the entries in
$$\left[\begin{array}{c|c|c} \bm I-\bm A_{jj}(\bm \varepsilon_0) &\hdots & \bm -\bm A_{jk}(\bm\varepsilon_0)\\ \hline \bm 0& \ddots & \vdots \\ \hline \bm 0& \bm 0&\bm I-\bm A_{kk}(\bm\varepsilon_0)\end{array}\right],\qquad\begin{bmatrix}\widetilde{\bm d}_j(\bm \varepsilon_0)\\ \hline \vdots \\ \hline  \widetilde{\bm d}_k(\bm \varepsilon_0) \end{bmatrix}.$$

\noindent As a consequence, so does the value of $\displaystyle{\frac{\partial \widetilde x_i}{\partial \epsilon_m}}$. Moreover, as a consequence of Equation \eqref{eq.derivada}, if $m$ corresponds to a position in either $\bm A_{rr}$ or $\widetilde{\bm d}_r$ the value of $\displaystyle{\frac{\partial \widetilde x_i}{\partial \varepsilon_m}}$ is 0 if $r<j$ and only depends on
$$\left[\begin{array}{c|c|c} \bm I-\bm A_{jj}(\bm \varepsilon_0) &\hdots & \bm -\bm A_{jr}(\bm\varepsilon_0)\\ \hline \bm 0& \ddots & \vdots \\ \hline \bm 0& \bm 0&\bm I-\bm A_{rr}(\bm\varepsilon_0)\end{array}\right],\qquad\begin{bmatrix}\widetilde{\bm d}_j(\bm \varepsilon_0)\\ \hline \vdots \\ \hline  \widetilde{\bm d}_r(\bm \varepsilon_0) \end{bmatrix}.$$

\noindent otherwise.

\end{rem}

\section*{Final Comments}

The study that it is done here for the system $(\bm I-\bm A)\bm x=\bm d$ is suitable to be generalized for systems of the type $(\lambda \bm I-\bm A)\bm x=\bm d$ for a fixed $\lambda$, corresponding to contracting or expanding economies. See \cite{Al} for more information concerning the model and \cite{G} for some results in this direction.



\section*{Disclosure statement}

The authors report there are no competing interests to declare.


\section*{Funding}

This work has been supported by the Spanish Agencia Estatal de Investigaci\'on thorugh project PID2020-116207GB-I00 and Junta de Comunidades de Castilla-La Mancha through project SBPLY/19/180501/000110.

\end{document}